\documentclass[a4paper, 12pt]{amsart}

\usepackage{amssymb,amsmath}
\usepackage[dvips]{graphics}
\usepackage[dvips]{color}
\usepackage{graphicx}
\usepackage[english]{babel}
\usepackage[utf8]{inputenc}
\usepackage{comment}
\usepackage{todonotes}

\usepackage{url}

\DeclareMathOperator*{\loc}{loc}

\newtheorem{theorem}{Theorem}[section]
\newtheorem{lemma}[theorem]{Lemma}

\newtheorem{corollary}[theorem]{Corollary}

\theoremstyle{definition}
\newtheorem{definition}[theorem]{Definition}

\newtheorem{remark}[theorem]{Remark}

\numberwithin{equation}{section}

\title[The equivalence of supersolutions]{The equivalence of weak and very weak supersolutions to the porous medium equation}

\thanks{The research is supported by the Emil Aaltonen Foundation and the Academy of Finland}

\author[Lehtel\"a and Lukkari]{Pekka Lehtel\"a and Teemu Lukkari}

\newcommand\rn{\mathbb R^n}
\newcommand\re{\mathbb R}
\newcommand\n{\mathbb N}

\newcommand\grad{\nabla}
\newcommand\bd{\partial}
\newcommand\ph{\varphi}
\newcommand\eps{\varepsilon}

\newcommand\supp{\operatorname{supp}}

\providecommand{\ch}[1]{\text{\raise 2pt \hbox{$\chi$}\kern-0.2pt}_{#1}}

\providecommand{\vint}[1]{\mathchoice
          {\mathop{\vrule width 5pt height 3 pt depth -2.5pt
                  \kern -9pt \kern 1pt\intop}\nolimits_{\kern -5pt{#1}}}%
          {\mathop{\vrule width 5pt height 3 pt depth -2.6pt
                  \kern -6pt \intop}\nolimits_{\kern -3pt{#1}}}%
          {\mathop{\vrule width 5pt height 3 pt depth -2.6pt
                  \kern -6pt \intop}\nolimits_{\kern -3pt{#1}}}%
          {\mathop{\vrule width 5pt height 3 pt depth -2.6pt
                  \kern -6pt \intop}\nolimits_{\kern -3pt{#1}}}}
                  
\begin{document}

\begin{abstract}
  We prove that various notions of supersolutions to the porous medium
  equation are equivalent under suitable conditions. More
  spesifically, we consider weak supersolutions, very weak
  supersolutions, and $m$-superporous functions defined via a
  comparison principle. The proofs are based on comparison principles
  and a Schwarz type alternating method, which are also interesting in
  their own right. Along the way, we show that Perron solutions with
  merely continuous boundary values are continuous up to the parabolic
  boundary of a sufficiently smooth space-time cylinder.
\end{abstract}

\subjclass[2010]{Primary 35K65, Secondary 35K20, 35D30, 35D99, 31C45}

\keywords{Porous medium equation, weak solutions, very weak solutions,
  supersolutions, comparison principle, boundary value problems}

\date{\today}  
\maketitle

\section{Introduction}
Our aim is to clarify and extend the connections between various
notions of solutions and supersolutions to the porous medium equation
\begin{equation}
  \label{eq:PME}
  u_t - \Delta u^m =0 \quad \text{in } \Omega_T=\Omega \times (0,T).
\end{equation}
We treat both the case of prescribed boundary values and the purely
local notions, and restrict our attention to the degenerate case
$m>1$.  For the basic theory of the equation and numerous further
references, we refer to the monographs \cite{daskalopoulos}, \cite{vazquez},
\cite{vazquez2} and \cite{nonlinear diffusion}.

There are at least two natural ways to define solutions to
\eqref{eq:PME}. \emph{Weak solutions} are defined by multiplying the
equation by a suitable test function and integrating by parts once. In
this definition, the function $u^m$ is assumed to be in a parabolic
Sobolev space.
In the case of the boundary value problem, the boundary values are
interpreted in a Sobolev sense.  The chief attraction of this notion
is that a weak solution itself is an admissible test function after a
mollification in the time direction, which leads to natural energy
estimates. On the other hand, we may integrate by parts twice in the
space variable, thus relaxing the regularity assumptions for
solutions. This leads to \emph{very weak solutions}, a notion which
makes sense under the minimal assumptions that $u$ and $u^m$ are
integrable.  The boundary values are taken into account via including
the appropriate integrals over the lateral boundary and at the initial
time.  One of the advantages of the very weak solutions is their
stability under convergence.  Weak and very weak solutions with fixed
boundary values turn out to be the same. This result is probably known
to experts, at least when the boundary values are sufficiently
regular.

It is important to understand not only the solutions, but also
supersolutions. Supersolutions arise naturally in obstacle problems
\cite{smallest, obstacle} and problems with measure data
\cite{measure, jee}. Furthermore, supersolutions connect the equation
to potential theory, providing important tools such as the Perron method
\cite{perron}. In the classical theory they also play a central role
in the study of boundary regularity, removability of sets and other
fine properties.

There are again various ways to define supersolutions. 
\emph{Weak and very weak supersolutions} (Definition \ref{def: weak
  super} and Definition \ref{def: very weak super}) satisfy the
inequality
\[
\frac {\partial u}{\partial t}-\Delta u^m \ge 0,
\]
the rigorous interpretation being analoguous to the concepts of weak
and very weak solutions.  Another way is to use a comparison
principle: supersolutions are lower semicontinuous functions which
satisfy a parabolic comparison principle with respect to continuous
weak solutions.  We call these supersolutions \emph{$m$-superporous
  functions} (Definition \ref{def: $m$-superporous}). This is one of
the ways to define superharmonic functions in classical potential
theory, and it is amenable to generalization to nonlinear equations.
In the case of the PME, the basic properties of this class of
supersolutions have been established in \cite{supersol}; see also
\cite{dichotomy}.  Several nice properties follow immetedially from
the definition of $m$-superporous functions. For instance, it is easy
to see that the minimum of two $m$-superporous functions is also
$m$-superporous. Moreover, the $m$-superporous functions form a closed
class under increasing convergence.

A natural question is whether the different classes of supersolutions
are equivalent. The similar problem is well understood in the case of
$p$-Laplace type equations, see \cite{superparabolic, LM}. However,
the question is more challenging for the porous medium equation. For
example, the boundary values cannot be perturbed in the standard way,
because constants cannot be added to solutions. Further difficulties
arise when trying to incorporate the very weak notions to the
arguments.  Therefore new methods have to be developed.

Our main result is the equivalence of the above classes of
supersolutions under suitable conditions:
\begin{theorem}\label{thm: super}
  The following properties are equivalent for continuous, nonnegative
  functions $u$.
  \begin{enumerate}
  \item  $u$ is a weak supersolution, \label{item:weak}
  \item $u$ is a very weak supersolution, \label{item:veryweak}
  \item $u$ is $m$-superporous. \label{item:superporo}
  \end{enumerate}
\end{theorem}
For completeness we also address the question of equivalence of the
classes of solutions, as it is difficult to find a reference where
this matter is treated thoroughly. Along the way, we obtain that
Perron solutions with merely continuous boundary values are continuous
up to the parabolic boundary of a sufficiently smooth space-time
cylinder, thus complementing the results of \cite{perron}.

The natural situation for Theorem \ref{thm: super} would be to
consider locally bounded lower semicontinuous functions. Indeed, lower
semicontinuity is the natural regularity of weak supersolutions, see
\cite{lower semicontinuity}, and local boundedness is definitely
necessary. Further, the equivalence of weak supersolutions and
$m$-superporous functions under these weaker assumptions has been
established in \cite{supersol}. Our contribution is including very
weak supersolutions to the theory. The necessity of boundedness can be
seen by considering the Barenblatt solution
\[
\mathcal{B}_m(x,t)=
\begin{cases}
  t^{-\lambda} \left ( C-\frac{\lambda(m-1)}{2mn}\frac{|x|^2}{t^{2\lambda/n}}\right )_+^{1/(m-1)}, \quad t>0,\\
0, \hfill t\le 0,
\end{cases}
\] 
where 
\[
\lambda=\frac{n}{n(m-1)+2}.
\]
The Barenblatt solution $\mathcal{B}_m$ is an unbounded
$m$-superporous function, but its gradient fails to be square
integrable in any neighbourhood of the origin, and thus
$\mathcal{B}_m$ is not a weak supersolution.

It is unclear whether the classes are the same, if one only assumes
lower semicontinuity. The crucial point where continuity is used is to
show that very weak supersolutions are also very weak supersolutions
with boundary values given by the function itself. This is needed for
proving the comparison principle for very weak
supersolutions. Further, there are other challenges already in the
continuous case. Therefore we find the continuity assumption
reasonable. Note that \emph{equivalence holds for solutions without
  assuming continuity:} nonnegative very weak solutions turn out to be
continuous after a redefinition on a set of zero measure, by
\cite{weak solutions}.

Yet another way to define supersolutions is viscosity supersolutions,
see \cite{brandle, caffarelli}. This notion uses pointwise touching
test functions. In this paper we focus on the previously mentioned
classes of supersolutions. A very interesting open question is whether
viscosity supersolutions are equivalent to the other notions of
supersolutions as well. The answer is known to affirmative for
equations similar to the $p$-Laplacian by \cite{JLM}, so one would
expect the same result to hold for the PME as well.

\section{Weak solutions}\label{sec: ws}

Throughout the work we use the following notation. We work in
space-time cylinders $\Omega_T=\Omega \times (0,T) \subset \re^{n+1}$,
where $\Omega\subset \rn$ is a bounded domain, such that $\bd \Omega$
is sufficiently nice, for example smooth or Lipschitz. We denote the
lateral boundary of $\Omega_T$ by $\Sigma_T=\bd \Omega \times [0,T]$
and the parabolic boundary by $\bd_p \Omega_T=\Omega \times \{0\} \cup
\Sigma_T$. For $U\Subset \Omega$, we denote $U_{t_1,t_2}=U\times
(t_1,t_2)$. The parabolic boundary of $U_{t_1,t_2}$ is defined as
$\bd_p U_{t_1,t_2}=U\times \{t_1\} \cup \bd U \times [t_1,t_2]$.

We consider the solutions to the boundary value problem
\begin{equation}
  \label{eq:BVP}
  \begin{cases}
  u_t-\Delta u^m = 0 \quad \text{on } \Omega_T,\\
u(x,0)=u_0(x),\\
u^m=g \quad \text{on } \Sigma_T,
\end{cases}
\end{equation}
where $u_0$ is in $H^1(\Omega)$ and $g$ is a continuous function
defined in $\overline{\Omega_T}$ such that $g\in
L^2(0,T;H^1(\Omega))$. Further, we require that the initial and
lateral boundary values are compatible in the sense that the function
$\varphi:\bd_p \Omega_T\to \re$ defined by
\begin{displaymath}
  \varphi(x,t)=
  \begin{cases}
    g(x,t)^{1/m}, & (x, t)\in \Sigma_T, \\
    u_0(x), & (x,t)\in \Omega\times \{0\}
  \end{cases}
\end{displaymath}
is continuous.  For simplicity, we will assume that $g$ and $u_0$ are
non-negative and thus the solutions will be non-negative as well by
the comparison principle, which will be proved in section \ref{sec:
  vws}. Hence we may assume that the solutions are always
non-negative. 


\begin{definition}\label{def: weak} 
We say $u$ is a local weak solution to \eqref{eq:PME} if $u^m\in
  L^2_{\loc}(0,T;H^1_{\loc}(\Omega))$ and $u$ satisfies the equality
  \[
  \int_{\Omega_T} (-u \ph_t + \nabla u^m\cdot \nabla \ph )\, dx \, dt = 0
  \]
  for any $\ph\in C_0^\infty (\Omega_T)$.

  A function $u$ is a weak solution to the boundary
  value problem \eqref{eq:BVP}, if $u^m\in
  L^2(0,T;H^1(\Omega))$, $u^m-g\in L^2(0,T;H_0^1(\Omega)),$
  and
  \[
  \int_{\Omega_T} (-u \ph_t + \nabla u^m\cdot \nabla \ph )\, dx \, dt = \int_\Omega u_0(x)\ph(x,0)\, dx
  \]
  for all smooth test functions $\ph$ with compact support in space,
  vanishing at the time $t=T$.
\end{definition}

We will show that the boundary value problem \eqref{eq:BVP} has at
most one weak solution.  This follows by using a clever test function
devised by Ole\u\i nik.

\begin{lemma}\label{weak solution lemma}
  Weak solutions to the boundary value problem \eqref{eq:BVP} are
 unique.
\end{lemma}
\begin{proof}
The proof is a standard application of the Ole\u\i nik test function
  \[
  \ph=
\begin{cases}
  \int_t^T (u^m -v^m) \, ds, \quad \text{if } 0\le t < T,\\
0 \quad \text{otherwise.}  
\end{cases}
\]
For a detailed proof, we refer to \cite[Theorem 5.3]{vazquez}.
\end{proof}

\section{Very weak solutions}\label{sec: vws}

In this section we consider another natural class of generalized
solutions, very weak solutions. This concept is defined as follows.
\begin{definition}\label{def: very weak}
  We say $u\in L^1_{\loc}(\Omega_T)$ is a local very weak solution to
  \eqref{eq:PME} if $u^m\in L^1_{\loc}(\Omega_T)$ and $u$ satisfies
  the equality
  \[
  \int_{\Omega_T} (u^m \Delta \eta + u \eta_t) \, dx \, dt =0
  \]
  for any $\eta\in C_0^\infty(\Omega_T).$

  A function $u\in L^1(\Omega_T)$ is a very weak solution to the
  boundary value problem \eqref{eq:BVP}, if $u^m \in L^1(\Omega_T)$
  and
  \[
  \int_{\Omega_T}( u^m \Delta \eta + u \eta_t) \, dx \, dt + \int_\Omega u_0(x)\eta (x,0) \, dx = \int_{\Sigma_T} g \partial_\nu \eta \, dS \, dt 
  \]
  for all smooth $\eta$ vanishing on $\Sigma_T$ and at time
  $t=T$. Note that the test functions $\eta$ are not required to have
  compact support in $\Omega_T$.
\end{definition}

We prove the comparison principle for the very weak solutions to the
boundary value problem \eqref{eq:BVP}. That is, if $u$ and $v$ are
very weak solutions to \eqref{eq:BVP} such that $u\ge v$ on $\bd_p
\Omega_T$ and $u^m, v^m \in L^2(\Omega_T)$, then $u\ge v$ in
$\Omega_T$. In fact, we only need to assume $u$ is a very weak
supersolution and $v$ is a very weak subsolution, see Lemma \ref{super
  comparison}. First, we present a technical lemma, which will be used in proving the comparison principles for very weak solutions and very weak supersolutions. The idea is, that in both cases the proof can be reduced to using the following lemma. 

\begin{lemma}\label{comparison tool}
  Let $u,v \in L^2(\Omega_T)$ and suppose $u^m,v^m\in L^2(\Omega_T)$. If
\[
\int_{\Omega_T}\big((v^m-u^m)\Delta \ph + (v-u)\ph_t \big) \, dx \, dt \ge 0
\]
for every smooth $\ph$ vanishing on $\Sigma_T$, then $u\ge v$ in $\Omega_T$. 
\end{lemma}
The proof of this lemma can be found in \cite[Theorem 6.5]{vazquez}. Next, we will show that the comparison principle for very weak solutions follows from this lemma.

\begin{lemma}\label{comparison lemma}
  Let $u,v \in L^2(\Omega_T)$ be very weak solutions to the boundary
  value problem \eqref{eq:BVP} with boundary and initial data $g, u_0$
  and $h,v_0$ respectively. Suppose that $u^m, v^m \in
  L^2(\Omega_T)$. If $u_0 \ge v_0$ in $\Omega$ and $g\ge h$ on
  $\Sigma_T$, then $u\ge v$ in $\Omega_T$.
\end{lemma}

\begin{proof}
By the definition of very weak solutions, 
\begin{align*}
  \int_{\Omega_T} \big (-u^m \Delta \ph - u \ph_t\big ) \,dx \, dt + \int_{\Sigma_T} g^m \partial_\nu \ph \,dS \, dt - \int_\Omega u_0(x) \ph(x,0) \,dx = 0
\end{align*}
and
\begin{align*}
  \int_{\Omega_T} \big (-v^m \Delta \ph - v \ph_t\big ) \,dx \, dt + \int_{\Sigma_T} h^m \partial_\nu \ph \,dS \, dt - \int_\Omega v_0(x) \ph(x,0) \,dx = 0
\end{align*}
for every smooth $\ph$ vanishing on $\Sigma_T$.  Subtracting the
equalities gives
\begin{align}\label{subtracting}
 & \int_{\Omega_T} \big ((v^m-u^m) \Delta \ph +(v- u) \ph_t\big) \,dx \, dt - \int_{\Sigma_T} (h^m-g^m) \partial_\nu \ph \,dS \, dt \nonumber\\
&+ \int_\Omega (v_0(x)-u_0(x)) \ph(x,0) \,dx = 0.\nonumber\\
\end{align}
Now suppose $\ph\ge 0 $ in $\Omega_T$. Since $u_0\ge v_0$ on $\Omega$, we see that 
\[
\int_\Omega (v_0-u_0) \ph \,dx\le 0.
\] 
The function $\ph$ vanishes on the lateral boundary $\Sigma_T$, so $\partial_\nu \ph \le 0$, and since $g\ge h$ on $\Sigma_T$, we have
\[
\int_{\Sigma_T} (h-g)\partial_\nu \ph \, dS \, dt\ge 0.
\]
Using the estimates above, we conclude
\begin{equation}\nonumber
\int_{\Omega_T} \big((v^m-u^m) \Delta \ph +(v- u) \ph_t\big) \,dx \, dt \ge 0.
\end{equation}
Now we may apply Lemma \ref{comparison tool} to conclude the proof. 
\end{proof}

\begin{remark}\label{finite union of cylinders}
  The comparison principle in Lemma \ref{comparison lemma} holds also
  for finite unions of space-time cylinders $K=\cup_{i=1}^N
  U^i_{t_1^i,t_2^i}.$ This can be proved by considering an enumeration
  $s_k$ of the times $t_j^i$, $i=1,\ldots, N$, $j=1,2$, where
  $s_1<s_2<\ldots <s_M$ and proving the result inductively for the
  sets $K\cap (\rn\times [s_i , s_{i+1}])$ using Lemma \ref{comparison
    lemma}.
\end{remark}

We use the following lemma from \cite{perron} to bypass the fact that
we may not add constants to solutions. We will present the proof for
the reader's convenience.

\begin{lemma}\label{boundary value lemma}
  Suppose $g$ is a continuous, non-negative function in $\overline{\Omega_T}$, such that $g\in L^2 (0,T; H^1(\Omega))$ and suppose $u_0\in H^1(\Omega)$ is non-negative. Let $\eps\in (0,1)$. Denote by $g_\eps = g + \eps$ and $u_{0,\eps}=u_0+\eps$. Let $u$ and $u_\eps$ be a weak solutions to \eqref{eq:BVP} with boundary and initial data $g, u_0$ and $g_\eps, u_{0,\eps}$ respectively. Then
\begin{align*}
  \int_{\Omega_T} (u_\eps-u)(u_\eps^m-u^m) \,dx \, dt \le \eps |\Omega_T| (M+1) + \eps|\Omega_T| (M+1)^m,
\end{align*}
where $M=\max \{\sup_{\overline{\Omega_T}} g, \sup_{\Omega} u_0 \}$. 
\end{lemma}

\begin{proof}
  
Since $u$ and $u_\eps$ are weak solutions, the  equalities 
\begin{align*}
  &\int_{\Omega_T} \big ( -u \ph_t + \nabla u^m \cdot \nabla \ph \big ) \, dx \, dt = \int_\Omega u_0(x) \ph(x,0) \, dx \quad \text{and}\\
 &\int_{\Omega_T} \big ( -u_\eps \ph_t + \nabla u_\eps^m \cdot \nabla \ph \big ) \, dx \, dt = \int_\Omega u_{0,\eps}(x) \ph(x,0) \, dx.
\end{align*}
hold.  Now a subtraction gives
\[
\int_{\Omega_T} \big ( -(u_\eps-u) \ph_t + \nabla (u_\eps^m-u^m) \cdot \nabla \ph \big ) \, dx \, dt = \int_\Omega (u_{0,\eps}-u_0) \ph(x,0) \, dx.
\]
We will use an Ole\u\i nik type test function defined as
\[
\ph(x,t)=
\begin{cases}
  \int_t^T (u_\eps^m-u^m-\eps) \, ds, \quad t\in[0,T),\\
0 \quad \text{otherwise}.
\end{cases}
\]
Now $\ph$ has the properties
\[
\ph_t = -(u_\eps^m - u^m-\eps) \quad \text{and} \quad \nabla \ph = \int_t^T \nabla (u_\eps^m -u^m) \, ds.
\] 
Thus
\begin{align*}
  &\int_{\Omega_T} \left ( (u_\eps-u)(u_\eps^m - u^m - \eps) + \nabla (u_\eps^m-u^m) \cdot \left(\int_t^T \nabla (u_\eps^m -u^m) \, ds\right) \right ) \, dx \, dt \\
&= \int_\Omega (u_{0,\eps}-u_0)  \left(\int_0^T (u_\eps^m-u^m-\eps)\,ds \right) \, dx.
\end{align*}
We observe that 
\begin{align*}
  &\int_{\Omega_T} \nabla (u_\eps^m-u^m) \cdot \left(\int_t^T \nabla (u_\eps^m -u^m) \, ds \right) \,dx \, dt \\
&= \frac{1}{2} \int_\Omega \left ( \int_0^T (\nabla u_\eps^m - \nabla u^m) \, ds \right )^2 \, dx \ge 0 \quad \text{and} \\
&-\eps T\int_\Omega (u_{0,\eps}-u_0) \, dx \le 0.
\end{align*}
Hence, we have the estimate
\begin{align}\label{boundaryvalueupperbound}
 & \int_{\Omega_T} (u_\eps - u) (u_\eps^m-u^m) \, dx \, dt \nonumber\\
&\le \eps \int_{\Omega_T} (u_\eps - u) \, dx \, dt + \int_\Omega (u_{0,\eps} - u_0 ) \left(\int_0^T (u_\eps^m - u^m) \, ds \right) \,dx.\nonumber\\
\end{align}
By the comparison principle $u\le M$ in $\Omega_T$ and thus by construction of $g_\eps$ and $u_{0,\eps}$, the comparison principle gives $u_\eps \le M+1$ in $\Omega_T$. Then the right hand side of \eqref{boundaryvalueupperbound} can be bounded from above using
\begin{align*}
 & u_\eps - u \le M+1, \\
&u_{0,\eps} - u_0 \le \eps \quad \text{and} \\
&u_\eps^m-u^m \le (u_\eps -u)^m \le (M+1)^m.
\end{align*}
We have
\[
   \int_{\Omega_T} (u_\eps - u) (u_\eps^m-u^m) \, dx \, dt 
\le \eps (M+1)|\Omega_T| + \eps (M+1)^m |\Omega_T|. \qedhere
\]
\end{proof}

For proving the equivalence of local weak and very weak
supersolutions, we need to consider solutions to the boundary value
problem \eqref{eq:BVP} when the functions $u_0$ and $g$ are only
assumed to be continuous. In such a case, the previous interpretation
of the boundary and initial conditions is no longer available, so we
use the notion of Perron solutions \cite{perron} instead. Perron
solutions are weak solutions in the interior, but the question whether
they attain the correct boundary values was left open in
\cite{perron}.  Next we show that this is indeed the case in
sufficiently smooth cylinders, by using a barrier argument. This
justifies calling the Perron solution the solution to the boundary
value problem~\eqref{eq:BVP}.



In order to construct a suitable lower barrier, we need to show the
existence of signed solutions to the boundary value problem with
smooth boundary values. This will be done in the next lemma. The proof
follows the ideas outlined in Chapter 5 of \cite{vazquez}.

\begin{lemma}\label{smooth boundary values}
  Let $\Omega_T=\Omega\times (0,T)$, where $\Omega\subset \rn$ is a
  bounded domain. Let $g$ be a smooth function defined in a
  neighbourhood of $\Sigma_T$ and let $u_0\in C^\infty (\Omega)$. Then
  there exists a weak solution to the boundary value problem
\[
\begin{cases}
  u_t-\Delta (|u|^{m-1} u) = 0 \quad \text{on } \Omega_T,\\
u(x,0)=u_0(x),\\
u^m=g \quad \text{on } \Sigma_T,
\end{cases}
\]
\end{lemma}

\begin{proof}
Let $\phi(s)=|s|^{m-1}s$. Define a smooth function $\phi_1$ such that
\[
\phi_1(s)=
\begin{cases}
  \phi(s) \quad \text{for } |s|\ge 1,\\
cs \hfill \text{for } |s|\le \frac{1}{2},\\
\end{cases}
\]
$\phi_1$ is convex for $s\ge 0$ and $\phi_1(-s)=-\phi_1(s)$. Now define $\phi_n(s)=n^{-m} \phi_1(n s)$, where $n=1,2,\ldots$. Then $\phi_n(s) = \phi(s)$ for $|s|\ge \frac{1}{n}$ and $\phi_n'(s) >0$ for every $s\in \re$. Moreover $\phi_n \rightarrow \phi$ uniformly on compact sets. We consider the approximate problem
\begin{equation}
  \label{eq:approximation}
  \begin{cases}
   u_t-\Delta \phi_n(u) = 0 \quad \text{on } \Omega_T,\\
u(x,0)=u_0(x),\\
\phi(u)=g \quad \text{on } \Sigma_T.
  \end{cases}
\end{equation}
By the quasilinear regularity theory (see \cite{quasilinear}), there exists a smooth solution $u_n$ to \eqref{eq:approximation}. Moreover, by the maximum principle we have
\[
-N \le u_n(x,t) \le M \quad \text{in } \Omega_T,
\]
where $N=\max\{\sup (-u_0), \sup(-g)\}$ and $M=\max\{\sup (u_0), \sup(g)\}$.
We multiply the equation $(u_n)_t-\Delta \phi_n(u_n)=0$ by a test function $\phi_n(u_n)-g\in L^2(0,T;H_ 0^1(\Omega))$ and integrate by parts to get
\[
\int_{\Omega_T} (u_n)_t (\phi_n(u_n)-g) \, dx \, dt + \int_{\Omega_T} \nabla \phi_n(u_n)\big ( \nabla \phi_n(u_n)-\nabla g\big )\,dx \, dt = 0. 
\]
Therefore
\begin{align}
  \label{eq:phigradient}
  \int_{\Omega_T} |\nabla \phi_n(u_n)|^2 \, dx \, dt =& \int_{\Omega_T} \nabla \phi_n(u_n) \cdot \nabla g \, dx \, dt - \int_{\Omega_T} (u_n)_t \phi_n(u_n)\, dx \, dt \nonumber \\
&+ \int_{\Omega_T} (u_n)_t g \, dx \, dt.\nonumber \\  
\end{align}
Let $\Psi_n$ denote the primitive of $\phi_n$, defined as 
\[
\Psi_n(s)=\int_0^s \phi_n(t) \, dt. 
\]
We observe that 
\[
(\Psi_n(u_n))_t = (u_n)_t \Psi_n' (u_n) = (u_n)_t \phi_n(u_n),
\]
and thus 
\begin{equation} \label{primitive}
\int_{\Omega_T} (u_n)_t \phi_n(u_n)\, dx \, dt = \int_\Omega \Psi_n(u_n(x,T)) \, dx  - \int_\Omega \Psi_n(u_0(x)) \, dx. 
\end{equation}
To control the last term on the right hand side of \eqref{eq:phigradient}, we integrate by parts to get
\begin{equation}\label{u_n time derivative}
 \int_{\Omega_T} (u_n)_t g \, dx \, dt = -\int_{\Omega_T} u_n g_t \, dx \, dt + \int_\Omega u_n(x,T) g(x,T)\, dx - \int_\Omega u_0(x) g(x,0)\, dx.
\end{equation}
Collecting the facts from \eqref{eq:phigradient}, \eqref{primitive} and \eqref{u_n time derivative} and using Young's inequality gives us an upper bound for the $L^2$-norm of the gradient
\begin{align*}
\int_{\Omega_T} | \nabla \phi_n(u_n)|^2 \, dx \, dt \le& C \bigg(\int_{\Omega_T}| \nabla g|^2 \, dx \, dt + \int_\Omega |\Psi_n(u_0(x))| \, dx \\
&+ \int_\Omega |\Psi_n(u_n(x,T))| \, dx+ \int_{\Omega_T} |u_n| |g_t| \, dx \, dt \\
&+ \int_\Omega |u_n(x,T)| |g(x,T)| \, dx + \int_\Omega |u_0(x)| |g(x,0)|\, dx \bigg).
\end{align*}
Thus $\nabla \phi_n(u_n)$ is uniformly bounded in $L^2(\Omega_T)$. In order to control the time derivative $(\phi_n(u_n))_t$, we multiply the equation $(u_n)_t-\Delta \phi_n(u_n)=0$ by the test function $\zeta(t)(\phi_n(u_n)-g)_t$, where $\zeta(t)$ is a smooth cut-off function, such that $0\le \zeta \le 1$, $\zeta(t)=1$ for $t\in(\eps,T-\eps)$, and $\zeta(0) = \zeta(T)=0$. Integrating by parts gives 
\[
\int_{\Omega_T} \zeta \big( \phi_n(u_n)-g\big)_t (u_n)_t \, dx \, dt = - \int_{\Omega_T} \zeta \nabla \phi_n(u_n)\cdot \nabla \big(\phi_n(u_n)-g\big)_t \, dx \, dt,
\]
which can be written as 
\begin{align}\label{phi_n time derivative}
\int_{\Omega_T} \zeta \phi_n(u_n)_t (u_n)_t \, dx \, dt =& \int_{\Omega_T} \zeta g_t u_t\, dx \, dt - \int_{\Omega_T} \zeta \nabla \phi_n(u_n)\cdot \nabla\phi_n(u_n)_t \, dx \, dt \nonumber\\ 
&+ \int_{\Omega_t} \zeta \nabla \phi_n(u_n)\cdot \nabla g \, dx \, dt \nonumber\\
=& I_1 + I_2 + I_3. \nonumber\\
\end{align}
We integrate $I_1$ by parts in the time variable to get
\[
I_1=\int_{\Omega_T} \zeta g_t u_t\, dx \, dt = - \int_{\Omega_T} \big(\zeta g_t\big)_t u_n \, dx \, dt.
\]
Integrating $I_2$ by parts gives
\begin{align*}
I_2=&- \int_{\Omega_T} \zeta \nabla \phi_n(u_n)\cdot \nabla \phi_n(u_n)_t \, dx \, dt \\=& \int_{\Omega_T} \big ( \zeta \nabla \phi_n(u_n)\big)_t \cdot \nabla \phi_n(u_n)\, dx \, dt \\
=& \int_{\Omega_T} \zeta' |\nabla \phi_n(u_n)|^2 \, dx \, dt 
+ \int_{\Omega_T} \zeta \nabla \phi_n(u_n) \cdot \nabla \phi_n(u_n)_t \, dx \, dt,
\end{align*}
and therefore
\[
I_2=\frac{1}{2}\int_{\Omega_T} \zeta' |\nabla \phi_n(u_n)|^2 \, dx \, dt.
\]
Finally, $I_3$ can be bounded by
\begin{align*}
|I_3|&=\left|\int_{\Omega_t} \zeta \nabla \phi_n(u_n)\cdot \nabla g \, dx \, dt\right|\\
&\le \left(\int_{\Omega_T} \zeta^2 |\nabla g|^2 \,dx \, dt\right )^{1/2}\left( \int_{\Omega_T} |\nabla \phi_n(u_n)|^2 \, dx \, dt\right)^{1/2}.
\end{align*}
Since $u_n$ is bounded, $\phi'_n(u_n)\le C$ for some $C$. Thus by \eqref{phi_n time derivative}, we get
\begin{align*}
  \int_{\Omega_{\eps,T-\eps}} |\phi_n(u_n)_t|^2 \, dx \, dt &=   \int_{\Omega_{\eps,T-\eps}} |(u_n)_t\phi'_n(u_n)|^2 \, dx \, dt \le  \int_{\Omega_T} \zeta|(u_n)_t^2\phi'_n(u_n)| \, dx \, dt\\
&\le C\big(|I_1| + |I_2| + |I_3|\big ).
\end{align*}
Hence, $\phi_n(u_n)_t$ is uniformly bounded in $L^2(\Omega \times (\eps, T-\eps))$. In conclusion, $\phi_n(u_n)$ is uniformly bounded in $H^1(\Omega\times (\eps, T-\eps))$. By compactness, there exists a subsequence $\phi_{n_j}(u_{n_j}) \rightarrow w \in L^2(\Omega\times (\eps, T-\eps))$ almost everywhere. It follows that $w \in L^2((0,T), H^1(\Omega))$. The sequence $u_{n_j}$ is uniformly bounded, so it converges to some $u$ almost everywhere (taking a subsequence, if necessary) and $\phi_{n_j}(u_{n_j}) \rightarrow \phi(u)$ almost everywhere. Therefore $w=\phi(u)$ almost everywhere. Since $u_n$ is a classical solution to  \eqref{eq:approximation}, it satisfies 
\[
\int_{\Omega_T} (-u_n \ph_t + \nabla \phi_n(u_n)\cdot \nabla \ph) \, dx \, dt = \int_\Omega u_0(x)\ph(x,0)\, dx.
\]
By weak compactness, $\nabla\phi_n(u_n)\rightarrow \nabla\phi(u)$ weakly, thus showing that indeed $u$ is a weak solution to the problem.
\end{proof}

We are now ready to show that Perron solutions attain the correct
boundary values in the classical sense. 


\begin{lemma}\label{existence lemma}
  Let the functions $g\in C(\overline{\Omega_T})$ and $u_0\in
  C(\overline{\Omega})$ be non-negative and compatible. Then the
  Perron solution to the boundary value problem \eqref{eq:BVP} attains
  the correct boundary values continuously.
\end{lemma}

\begin{proof}
  
  We will show the claim by a barrier type argument. To simplify
  notation we write
  \[
  \ph(x,t)=
  \begin{cases}
    g(x,t)^{1/m}, \quad (x,t)\in \Sigma_T,\\
    u_0(x), \quad (x,t)\in \Omega \times \{0\}.
  \end{cases}
  \]
  Fix $\xi \in \bd_p \Omega_T$ and take $\eps>0$. We will show that
  there exists a supersolution $v^+ \in \mathfrak{U}_\ph$, such that
  $\lim_{z\rightarrow \xi}v^+(z)=\ph(\xi)+\eps$ and a subsolution $v^-
  \in \mathfrak{L}_\ph$, such that $\lim_{z\rightarrow
    \xi}v^-(z)=\ph(\xi)-\eps$. Here $\mathfrak{U}_\ph$ and
  $\mathfrak{L}_\ph$ denote the upper and lower Perron classes
  respectively.

  The upper barrier $v^+$ can be constructed by solving the boundary
  value problem \eqref{eq:BVP} with boundary values $\ph+\eps$. A
  continuous solution exists by the quasilinear theory, as described
  in the proof of the previous lemma. Moreover $v^+$ is continuous up
  to the boundary by \cite{ziemer}. In order to construct the lower
  barrier $v^-$, we will consider a small neighbourhood $E$ of
  $\xi$. Let $f$ be a smooth function, such that
  $f(\xi)=\ph(\xi)-\eps$ and $f=-k$ on $\bd_p(E\cap \Omega_T)$ outside
  a neighbourhood of $\xi$. By Lemma \ref{smooth boundary values},
  there exists a weak solution $\tilde{v}$ in $E\cap \Omega_T$ with
  boundary values $f$. We extend $\tilde{v}$ to the whole $\Omega_T$
  by defining
  \[
  v^- =
  \begin{cases}
    \max \{\tilde{v}, -k\} \quad \text{in } E\cap \Omega_T,\\
    -k \hfill \text{in } \Omega_T\setminus E.
  \end{cases}
  \]
  By choosing $k$ large enough, we have $v^- \in \mathfrak{L}_\ph$ and
  $v^-=\tilde{v}$ in $E\cap \Omega_T$. Again, the continuity of $v^-$
  up to the boundary is provided by \cite{ziemer}.

  By the definition of the Perron solution, $v^- \le u \le v^+$ and
  thus
  \begin{displaymath}
    \ph(\xi)-\eps \le \liminf_{z\rightarrow \xi} u(z) \le
    \limsup_{z\rightarrow \xi} u(z)\leq \ph(\xi)+\eps.
  \end{displaymath}
  Since this holds for every $\eps>0$, we conclude that
  $\lim_{z\rightarrow \xi} u(z)=\ph(\xi)$. 
\end{proof}

We are now ready to prove the first of the main results, the
equivalence of the different notions of solutions to the boundary
value problem. We emphasize the fact that the boundary and initial
values are only assumed to be continuous.
\begin{theorem}\label{thm: solutions}
  Let $u$ be the Perron solution and $v$ a very weak solution to the
  boundary value problem
  \[
  \begin{cases}
    u_t-\Delta u^m = 0 \quad \text{on } \Omega_T,\\
    u(x,0)=u_0(x),\\
    u^m=g \quad \text{on } \Sigma_T,
  \end{cases}
  \]
  with continuous, compatible boundary values $u_0$ and $g$.  If
  $v^m\in L^2(\Omega_T)$, then $u=v$.
\end{theorem}
\begin{proof}
  The claim follows from the comparison principle for very weak
  solutions (Lemma \ref{comparison lemma}) as soon as we show that the
  Perron solution $u$ is also a very weak solution to the boundary
  value problem.  For smooth boundary values, this follows by Green's
  formula from the fact that the Perron solution also attains the
  correct boundary values in the Sobolev sense, see Theorem 5.8 in
  \cite{perron}.

  It remains to reduce the general case to the smooth case. We do this
  by an approximation argument.  Define as before
    \[
  \ph(x,t)=
  \begin{cases}
    g(x,t)^{1/m}, \quad (x,t)\in \Sigma_T,\\
    u_0(x), \quad (x,t)\in \Omega \times \{0\},
  \end{cases}
  \]
  extend $\ph$ continuously to the whole space and choose smooth
  functions $\ph_j$ converging to $\ph$ uniformly and such that
  \begin{displaymath}
    \ph_j\leq \ph \leq \ph_j +1/j.
  \end{displaymath}
  Further, let $u_j$ and $v_j$ be the Perron solutions with boundary
  values $\ph_j$ and $\ph_j +1/j$, respectively. Since $u_j\leq u$
  and 
  \begin{displaymath}
    u-u_j\leq v_j-u_j\to 0
  \end{displaymath}
  as $j\to \infty$ by Lemma \ref{boundary value lemma}, we have that
  $u_j\to u$ pointwise in $\Omega_T$. Now $u_j$ is a very weak
  solution to the boundary value problem with boundary values given by
  $\ph_j$, and passing to the limit $j\to \infty$ in the very weak
  formulation for $u_j$ shows that $u$ is a very weak solution to the
  boundary value problem with boundary values given by $\ph$.
\end{proof}


The previous theorem together with the continuity result in \cite{weak
  solutions} implies the equivalence of local weak and very weak
solutions.

\begin{corollary}
  A nonnegative function $u$ is a local very weak solution to the PME
  if and only if $u$ is a local weak solution to the PME.
\end{corollary}

\begin{proof}
  By \cite{weak solutions}, local very weak solutions are continuous
  in the interior of $\Omega_T$. Thus, following the proof of Lemma
  \ref{boundary values} below, we may show that local very weak
  solutions are solutions to the boundary value problem \eqref{eq:BVP}
  in space-time cylinders $B_{t_1,t_2}\Subset \Omega_T$ where the base
  is a ball, with boundary values defined by the function itself.
  Therefore the result follows from Theorem \ref{thm: solutions} and
  the fact that being a weak solution is a local property.
\end{proof}

\section{Supersolutions}\label{sec: ss}

In this section, we turn our attention to supersolutions.  The
definitions of weak supersolutions and very weak supersolutions are
analogous to those of weak solutions and very weak solutions.

\begin{definition}\label{def: weak super}
  A function $u\in L^2_{\loc}(0,T;H^1_{\loc}(\Omega))$ is a (local)
  weak supersolution to \eqref{eq:PME} if $u^m\in
  L^2_{\loc}(0,T;H^1_{\loc}(\Omega))$
  \[
  \int_{\Omega_T}\big (-u \ph_t+\grad u^m \cdot \grad \ph \big )\,dx \, dt \ge 0
  \]  
  for all non-negative, compactly supported smooth test functions
  $\ph$.
\end{definition}
As in the case of weak solutions, it is natural to consider also 
very weak supersolutions.
\begin{definition}\label{def: very weak super}
  A function $u\in L^1_{\loc}(\Omega_T)$ is a (local) very weak
  supersolution to \eqref{eq:PME}, if $u^m\in L^1_{\loc}(\Omega_T)$
  and
  \[
  \int_{\Omega_T}\big (-u\ph_t - u^m \Delta\ph\big) \,dx \, dt \ge 0
\]
for all non-negative, compactly supported smooth test functions $\ph$. 
\end{definition}

As the first step in relating the various classes of supersolutions we
will show that continuous very weak supersolutions can be seen as
supersolutions to the boundary value problem in a space-time cylinder,
whose base is a ball, with boundary values defined by the function
itself. The known argument for solutions (see
e.g. \cite{daskalopoulos}) carries over to supersolutions without
serious difficulties. However, the continuity assumption is essential
in the proof.

\begin{lemma}\label{boundary values}
  Let $u$ be a non-negative, continuous very weak supersolution in
  $\Omega_T$. Then for any $B_r\times (t_1,t_2)\Subset \Omega_T$, $u$
  is a very weak supersolution in $B_r\times (t_1,t_2)$ with boundary
  values $u\big |_{\bd_p B_r\times (t_1,t_2)}$.
\end{lemma}
\begin{proof}
  Let $\eta$ be a smooth function in $\overline{B_r\times (t_1,t_2)}$
  vanishing on $\bd B_r\times (t_1,t_2).$ For $\eps\in (0,r)$ and
  $\theta \in [0,\eps)$ let $\Psi_{\eps \theta}$ be the radial,
  continuous function satisfying
  \[
  \Psi_{\eps \theta}(\rho)=
  \begin{cases}
    1 \quad \mbox{for $0 \le \rho \le r-\eps$,}\\
    0 \quad \mbox{for $\rho \ge r-\theta$,}\\
  \end{cases}
  \]
  and
  \[
  \Delta \Psi_{\eps \theta}(x)=\frac{n-1}{|x|}\Psi'_{\eps \theta}(|x|)+\Psi''_{\eps \theta}(|x|)= 0 \quad \mbox{in $B_{r-\theta}\setminus B_{r-\eps}$}.
  \]
By solving the equation we obtain
\begin{align*}
  &\Psi_{\eps \theta}(\rho)= \frac{\rho}{\theta-\eps}+ \frac{r-\theta}{\eps-\theta}, \quad n=1,\\
  &\Psi_{\eps \theta}(\rho)=
  \frac{\ln(\rho)-\ln(r-\theta)}{\ln(r-\eps)-\ln(r-\theta)}, \quad n=2, \\
  &\Psi_{\eps \theta}(\rho)=
  \frac{(r-\eps)^n(\rho^{n-2}(r-\theta)^2-(r-\theta)^n)}{(r-\eps)^n(r-\theta)^2-(r-\eps)^2(r-\theta)^n} \frac{1}{\rho^{n-2}},\quad n>2.
\end{align*}
From now on, we will assume that $n>2$ for simplicity.  A similar
reasoning can be carried out also in the cases $n=1,2$. We observe
that
\[
\grad \Psi_{\eps \theta}(x)=\begin{cases}\frac{(n-2)(r-\eps)^{n-2}(r-\theta)^{n-2}}{(r-\eps)^{n-2}-(r-\theta)^{n-2}} \frac{x}{|x|^n}= -W_{\eps \theta} \frac{x}{|x|^n} \quad \mbox{in $B_{r-\theta}\setminus B_{r-\eps}$},\\
  0 \quad \mbox{otherwise.}
\end{cases}
\]
Now $\Delta \Psi_{\eps \theta}$ can be seen as the distribution
\[
\int_B \ph \Delta \Psi_{\eps \theta} \,dx = W_{\eps \theta} \left ( \int_{\bd B_{r-\theta}} \ph \, dS - \int_{\bd B_{r-\eps}} \ph \, dS \right ). 
\]

Let $K_\nu$ be a standard mollifier, i.e. a smooth, positive, radially
symmetric function supported in $B_\nu(0)$ with the property $\int
K_{\nu} \,dx=1$. Define
\[
\Psi_{\eps \theta}^\nu(x)=\Psi_{\eps \theta}*K_\nu (x).
\]
Let $\phi_\lambda(t)$ be smooth functions with compact support in
$(0,T)$, converging to
\[
H_{t_1}(t)=\begin{cases} 0 \quad \mbox{for $t<t_1$},\\
1 \quad\mbox{for $t\ge t_1$},
\end{cases}
\]
as $\lambda\rightarrow 0$.

Define $\ph(x,t)=\Psi_{\eps \theta}^\nu(x)\phi_\lambda(t)\eta(x,t).$
Now $\ph$ is a smooth, compactly supported function in $\Omega_T$ and
thus
\begin{equation}
  \label{vwsupsol}
  \int_{\Omega_T}\big (-u\ph_t - u^m \Delta\ph \big) \,dx \, dt \ge 0.
\end{equation}
Since 
\begin{align*}
  \Delta\ph &= \phi_\lambda (\Psi_{\eps \theta}^\nu \Delta \eta+2\grad \Psi_{\eps \theta}^\nu\cdot \grad \eta + \eta \Delta \Psi_{\eps \theta}^\nu) \quad \mbox{and}\\
\ph_t&=\Psi_{\eps \theta}^\nu(\phi_\lambda \eta_t + (\phi_\lambda)_t \eta),
\end{align*}
inequality \eqref{vwsupsol} can be written as
\begin{align}\nonumber
  0 \le& \int_{t_1}^{t_2}\int_B \big(-u^m\phi_\lambda \Psi_{\eps \theta}^\nu \Delta \eta - u \Psi_{\eps \theta}^\nu \phi_\lambda \eta_t \big)\,dx \, dt\\
&- \int_{t_1}^{t_2}\int_B 2 u^m \phi_\lambda \grad \Psi_{\eps \theta}^\nu\cdot \grad \eta \,dx \, dt\nonumber\\
&-\int_{t_1}^{t_2}\int_B u^m \phi_\lambda \eta \Delta \Psi_{\eps \theta}^\nu \,dx \, dt\nonumber\\
&-\int_{t_1}^{t_2}\int_B u \Psi_{\eps \theta}^\nu (\phi_\lambda)_t \eta \,dx \, dt\nonumber\\
=& I_1+I_2+I_3+I_4.\nonumber\\
\label{supersolution boundary value}
\end{align}
Letting $\nu\rightarrow 0$, $\theta\rightarrow 0$, $\eps \rightarrow
0$ and $\lambda \rightarrow 0$ gives us
\[
I_1\rightarrow \int_{t_1}^{t_2}\int_B \big(-u^m \Delta \eta - u\eta_t \big) \,dx \, dt.
\]
Letting $\nu \rightarrow 0$ and $\lambda\rightarrow 0$ and taking $\supp(\grad \Psi_{\eps \theta})$ into account, 
\begin{align*}
I_2&=-2\int_{t_1}^{t_2}\int_{B_{r-\theta} \setminus B_{r-\eps}} u^m \grad \Psi_{\eps \theta}\cdot \grad \eta  \,dx \, dt \\
&=2 W_{\eps \theta} \int_{t_1}^{t_2}\int_{B_{r-\theta} \setminus B_{r-\eps}} u^m \frac{x}{|x|^n}\cdot \grad \eta \,dx \, dt \\
&=2 W_{\eps \theta} \int_{t_1}^{t_2}\int_{r-\eps}^{r-\theta} \int_{S^{n-1}} u^m \partial_\nu \eta\big |_{|x|=\rho} \,dS \, d\rho \, dt\\
&\rightarrow 2 W_{\eps 0} \int_{t_1}^{t_2}\int_{r-\eps}^{r} \int_{S^{n-1}} u^m \partial_\nu \eta\big |_{|x|=\rho} \,dS \, d\rho \, dt\\
\end{align*}
as $\theta \rightarrow 0$. Since
\[
W_{\eps 0}=\frac{(n-2)(r-\eps)^{n-2}r^{n-2}}{r^{n-2}-(r-\eps)^{n-2}}=\frac{(r-\eps)^{n-2}r^{n-2}}{\eps \xi^{n-3}},
\]
where $\xi\in(r-\eps, r)$, we get
\begin{align*}
&2 W_{\eps 0} \int_{t_1}^{t_2}\int_{r-\eps}^{r} \int_{S^{n-1}} u^m \partial_\nu \eta\big |_{|x|=\rho} \,dS \, d\rho \, dt \\
&= 2 \frac{(r-\eps)^{n-2}r^{n-2}}{\xi^{n-3}} \int_{t_1}^{t_2}\frac{1}{\eps}\int_{r-\eps}^{r} \int_{S^{n-1}} u^m \partial_\nu \eta\big |_{|x|=\rho} \,dS \, d\rho \, dt\\
&\rightarrow 2 r^{n-1}\int_{t_1}^{t_2} \int_{S^{n-1}} u^m \partial_\nu \eta\big |_{|x|=r} \,dS \, dt.
\end{align*}
Since $u$ is continuous,  
\[
I_3\rightarrow - W_{\eps \theta}\int_{t_1}^{t_2}\left (\int_{\bd B_{r-\theta}} u^m  \eta \,dS-\int_{\bd B_{r-\eps}} u^m  \eta \,dS \right ) \, dt, 
\]
as $\nu\rightarrow 0$ and $\lambda\rightarrow 0$, due to the weak
convergence of the measures $\Delta \Psi_{\eps \theta}^\nu$. Note that
the continuity assumption is essential here, as we use weak
convergence for signed measures.  Now, as $\theta \rightarrow 0$ we
get
\begin{align*}
&-W_{\eps 0} \int_{t_1}^{t_2} \int_{S^{n-1}} u^m \left ( \eta \big |_{|x|=r} -\eta \big |_{|x|=r-\eps} \right ) \,dS \, dt \\
&=-\frac{(r-\eps)^{n-2}r^{n-2}}{\xi^{n-3}}\int_{t_1}^{t_2} \int_{S^{n-1}} u^m \left (\frac{ \eta \big |_{|x|=r} -\eta \big |_{|x|=r-\eps}}{\eps}\right) \,dS \, dt\\
&\rightarrow -r^{n-1} \int_{t_1}^{t_2} \int_{S^{n-1}} u^m \partial_\nu \eta \big |_{|x|=r} \,dS \, dt.
\end{align*}
Finally, letting $\nu\rightarrow 0, \theta \rightarrow 0 $ and $\eps
\rightarrow 0$ gives us
\begin{align*}
I_4&\rightarrow -\int_{t_1}^{t_2} \int_B u (\phi_\lambda)_t \eta \,dx \,dt \rightarrow -\int_B u(x,0) \eta(x,0)\,dx
\end{align*}
as $\lambda \rightarrow 0$. Now we may conclude from inequality
\eqref{supersolution boundary value} that
\begin{align*}
&  \int_{t_1}^{t_2}\int_B \big(-u^m \Delta \eta - u \eta_t \big)\,dx \, dt + \int_{t_1}^{t_2}\int_{\bd B} u^m \partial_\nu \eta \,dS \, dt \\
&- \int_B u(x,0) \eta(x,0) \,dx \ge 0. \qedhere
\end{align*}
\end{proof}

The next step is to show that continuous very weak supersolutions
satisfy the comparison principle with continuous very weak solutions
in the special case where we look at a cylinder whose base is a
ball. Since weak solutions are also very weak solutions, this lemma is
the key to showing that continuous very weak supersolutions are indeed
$m$-superporous functions in the sense of Definition \ref{def:
  $m$-superporous} below.
\begin{lemma}\label{super comparison}
  Let $u$ be a continuous very weak supersolution and let $v$ be a continuous very weak solution in $\Omega_T$. Let $U_{t_1,t_2}=B_r\times [t_1,t_2]\Subset \Omega_T$. Then if $u\ge v$ on $\bd_p U_{t_1,t_2}$, then $u\ge v$ in $U_{t_1,t_2}$.  
\end{lemma}
\begin{proof}
Since $u$ is a continuous very weak supersolution, Lemma \ref{boundary values} gives
\begin{align*}
 & \int_{t_1}^{t_2}\int_B \big (-u^m \Delta \ph - u \ph_t\big ) \,dx \, dt + \int_{t_1}^{t_2}\int_{\bd B} u^m \partial_\nu \ph \,dS \, dt \\
&- \int_B u(x,0) \ph(x,0) \,dx \ge 0
\end{align*}
for every smooth $\ph$ vanishing on $\bd B_r \times (t_1,t_2)$. By definition of very weak solutions
\begin{align*}
&\int_{t_1}^{t_2}\int_B \big (-v^m \Delta \ph - v \ph_t\big ) \,dx \, dt + \int_{t_1}^{t_2}\int_{\bd B} v^m \partial_\nu \ph \,dS \, dt\\  &- \int_B v(x,0) \ph(x,0) \,dx = 0.
\end{align*}
Subtracting the inequalities gives
\begin{align*}
 & \int_{t_1}^{t_2}\int_B \big ((v^m-u^m) \Delta \ph +(v- u) \ph_t\big) \,dx \, dt - \int_{t_1}^{t_2}\int_{\bd B} (v^m-u^m) \partial_\nu \ph \,dS \, dt \\
&+ \int_B (v(x,0)-u(x,0)) \ph(x,0) \,dx \ge 0.
\end{align*}
In fact, we could have assumed $v$ is only a very weak subsolution to
get the same inequality. Now we are at similar situation as in
\eqref{subtracting} with inequality instead of equality. However, the
same reasoning still applies, and thus we may use
Lemma \ref{comparison tool} to conclude that $u\ge v $ in
$U_{t_1,t_2}$. Note that $u$ and $v$ are continuous functions and thus
$u,v,u^m,v^m \in L^2(U_{t_1,t_2})$ so the assumptions of Lemma
\ref{comparison lemma} hold for $u$ and $v$.
\end{proof}

The following lemma extends the comparison property to finite unions
of space-time cylinders whose bases are balls.  We utilize a Schwarz
type alternating method. The proof is delicate since we need to work
around the fact that constants cannot be added to solutions.


\begin{lemma} \label{comparison extends} Let $B_i\subset \rn$,
  $i=1,\ldots,N$ be a collection of balls and let $U_i = B_i \times
  (t_1,t_2)$. Set $K=\cup_{i=1}^N U_i$. Suppose that $u$ satisfies the
  comparison principle for cylinders whose base is a ball in a
  neighbourhood of $\overline{K}$. That is, if $h$ is a continuous
  weak solution such that $h\le u$ on $\bd_p U$, where $U$ is a
  cylinder whose base is a ball, then $h\le u$ in $U$.

  Then the comparison principle for $u$ holds also in $K$. That is, if
  $h$ is a solution of the PME in $K$, which is continuous up to the
  boundary of $K$, then $h\le u$ on $\bd_p K$ implies $h\le u$ in $K$.
\end{lemma}

\begin{proof}
  Let $\delta >0$. Take $\ph \in C^\infty (K)\cap C(\overline{K})$ such that
  \begin{align*}
    \ph &\le u \quad \text{in } K\cup \bd_p K,\\
h -\delta &\le \ph \quad \text{on } \bd_pK.
  \end{align*}
Let $\Psi_0$ be a continuous weak subsolution to the PME in $K$ satisfying 
\begin{align*}
  \Psi_0 &= \ph \quad \text{on } \bd_p  K,\\
\Psi_0&\le \ph \quad \text{in } K.
 \end{align*}
 Such a subsolution can be constructed by the arguments leading to
 Theorem 2.6 of \cite{obstacle}. We want to construct an increasing
 sequence of continuous weak subsolutions $v_k$ such that
 $v_k\rightarrow w$, where $w$ is a continuous weak solution. Set
 $v_0=\Psi_0$. For $1\le i \le N$ and $j\ge 0$ we define the functions
 recursively by
\[
v_{Nj+i}=
\begin{cases}
  \tilde{v}_{Nj+i-1} \quad \text{in }U_i,\\
v_{Nj+i-1} \quad \text{in } K\setminus U_i,
\end{cases}
\] 

where $\tilde{v}_{Nj+i-1}$ is the continuous weak solution in $U_i$
with boundary values $v_{Nj+i-1}$ on $\bd_p U_i$. Existence and
continuity of $\tilde{v}_{Nj+i-1}$ are provided by \cite{perron}. Thus
$v_k$ is a continuous weak subsolution for each $k$.

We want to show that the sequence $v_k$ converges to a continuous weak
solution. Since $v_{Nj+i-1}$ is a continuous weak subsolution in $U_i$
and $v_{Nj+i}$ is a continuous weak solution in $U_i$, we may use the
comparison principle for subsolutions in $U_i$. By construction,
$v_{Nj+i-1}$ and $v_{Nj+i}$ coincide on $\bd_p U_i$ and thus
$v_{Nj+i-1} \le v_{Nj+i}$ in $U_i$. Hence $v_k$ is an increasing
sequence.

The function $v_0$ has been chosen in such a way, that $v_0 \le \ph
\le u$ in $K$. Suppose $v_{Nj+i-1} \le u$ in $K$. Now $v_{Nj+i} \le u$
on $\bd_p U_i$ by construction and therefore since $u$ satisfies the
comparison principle in $U$ by assumption, $v_{Nj+i} \le u$ in
$U_i$. It follows by induction, that $v_k\le u$ in $K$ for all
$k$. Now $v_k$ is bounded and increasing and thus $v_k\rightarrow w
\le u$ for some $w$ in $\overline{K}$.

Weak solutions are locally H\"older continuous (see \cite{non-negative
  solutions}), so for each $z\in K$, there are $i_z\in \n$ and $r_z>0$
such that $B(z,r_z) \subset U_{i_z}$ and $v_{Nj+i_z}$ is H\"older
continuous in $B(z,r_z)$ for every $j$. Therefore the subsequence
$v_{Nj+i_z}$ converges to a continuous function in $B(z,r_z)$. Since
$v_k \rightarrow w$, we conclude that $w$ is continuous in $K$. To
show the continuity of $w$ up to the boundary, let $h_\ph$ be the
continuous weak solution in $K$, with boundary values $\ph$ on $\bd_p
K$. By construction $v_0 \le w$ and by the comparison principle $w\le
h_\ph$ in $K\cup \bd_p K$. Since $v_0=h_\ph$ on $\bd_p K$ and $v_0,
h_\ph$ are continuous, we conclude that $w$ is continuous in $K\cup
\bd_p K$.

Finally, we need to show that $w$ is indeed a continuous weak solution in $K$. It suffices to show that $w$ is a continuous weak solution in $U_i^\rho=B_{i_0}\times (t_1,t_2-\rho)$ for every $1\le i_0 \le N$ and $\rho >0$. The sequence $v_k$ is increasing, each $v_k$ is continuous in $\overline{K}$ and $w$ is continuous in $K\cup\bd_p K$. Therefore $v_k\rightarrow w$ uniformly in $K\cap \{t\le t_2-\rho\}$. Thus for every $\eps>0$ there is $j_\eps$ such that for $j\ge j_\eps$, we have
\[
|w-v_{Nj+i_0}|<\eps \quad \text{in } \overline{U_{i_0}^\rho}.
\]
Let $w'$ be a continuous weak solution in $U_{i_0}^\rho$ with boundary values $w$ on $\bd_pU_{i_0}^\rho$. Now 
\[
 v_{Nj+i_0}\le w'+\eps, \quad \text{and} \quad
w'\le v_{Nj+i_0}+\eps \quad \text{on } \bd_p U_{i_0}^\rho.
\]

Let $w_\eps$ be the continuous weak solution in $U_{i_0}^\rho$ with boundary values $w'+\eps$ on $\bd_pU_{i_0}^\rho$ and let $v_\eps$ be the continuous weak solution in $U_{i_0}^\rho$ with boundary values $v_{Nj+i_0}+\eps$ on $\bd_p U_{i_0}^\rho$. Then by the comparison principle for weak solutions $v_{Nj+i_0}\le w_\eps$ and $w'\le v_\eps$ in $U_{i_0}^\rho$. Since
\begin{align*}
|w'-v_{Nj+i_0}|&=(w'-v_{Nj+i_0})_+ + (v_{Nj+i_0}-w')_+ \\
&\le (w_\eps -v_{Nj+i_0})+(v_\eps -w'),
\end{align*}
we may use Lemma \ref{boundary value lemma} to conclude $|v_{Nj+i_0}- w'|\rightarrow 0$ uniformly in $U_{i_0}^\rho$ (passing to a subsequence, if necessary). Therefore we may assume
\[
v_{Nj+i_0}-\eps \le w' \le v_{Nj+i_0} + \eps \quad \text{in }\overline{U_{i_0}^\rho}. 
\]
Now
\[
|w-w'|\le |w-v_{Nj+i_0}| + |v_{Nj+i_0}-w'| < 2\eps \quad \text{in } \overline{U_{i_0}^\rho}. 
\]
Letting $\eps\rightarrow 0$ shows that $w'=w$ and thus $w$ is a continuous weak solution in $U_{i_0}^\rho. $ Denote by $w_\delta$ the continuous weak solution in $K$ with boundary values $\ph+\delta$ on $\bd_p K$. Then $w_\delta \ge h $ on $\bd_p K$ and thus by comparison principle for the continuous weak solutions, the inequality holds in $K$. Therefore
\[
0\le (h-w)_+(h^m-w^m)_+ \le (w_\delta-w)(w_\delta^m-w^m).
\] 
Lemma \ref{boundary value lemma} gives us 
\begin{align*}
0 \le \int_K (h-w)_+(h^m-w^m)_+ \,dx \,dt &\le \int_K (w_\delta-w)(w_\delta^m-w^m) \,dx \, dt \\
&\le \delta |K| (\sup \ph +1)+\delta |K| (\sup \ph +1)^m. 
\end{align*}
Since this holds for any $\delta>0$, letting $\delta \rightarrow 0$ shows that $h\le w$ in $K$. On the other hand, $w\le u$ by construction and thus $h\le u$ as we wanted.
\end{proof}

\section{$m$-superporous functions}\label{sec: sp}
Another important class of supersolutions is the class of
$m$-superporous functions, defined in terms of a comparison principle
with respect to continuous weak solutions. This class is analoguous to
superharmonic functions in classical potential theory, where the
definition is due to Riesz.

\begin{definition}\label{def: $m$-superporous}
A function $u: \Omega_T \rightarrow [0,\infty]$ is \emph{$m$-superporous}, if
\begin{itemize}
\item [(1)] $u$ is lower semicontinuous,
\item [(2)] $u$ is finite in a dense subset of $\Omega_T$, and
\item [(3)] the following parabolic comparison principle holds: Let
  $U_{t_1,t_2} \Subset \Omega_T$, and let $h$ be a weak solution to
  the PME which is continuous in $\overline{U_{t_1,t_2}}$. Then, if
  $h\le u$ on $\bd_p U_{t_1,t_2}$, $h\le u$ also in $U_{t_1,t_2}$.
\end{itemize}
\end{definition}

Our aim in this section is to is to connect $m$-superporous functions
to the notions of weak and very weak supersolutions, i.e. to prove
Theorem \ref{thm: super}.  The first step is the next lemma, which
shows that continuous very weak supersolutions are
$m$-superporous. This is essentially a consequence of
Lemma~\ref{comparison extends}, but some care is again required due to
the fact that constants may not be added to solutions.

\begin{lemma}\label{very weak super to $m$-superporous}
  Let $u$ be a continuous very weak supersolution to \eqref{eq:PME} in
  $\Omega_T$. Then $u$ is $m$-superporous.
\end{lemma}
\begin{proof}
  Let $U_{t_1,t_2} \Subset \Omega_T$ and let $h$ be a continuous weak
  solution such that $h\le u$ on $\bd_p U_{t_1,t_2}$. We want to show,
  that $h\le u$ in $U_{t_1,t_2}$. Take $\eps >0$ and define the set
  \[
  D=\{(x,t)\in \overline{U_{t_1,t_2}} : h \ge u+\eps\}.
  \]
  Now $D$ is compact and by the assumption $D\subset U\times
  [t_1,t_2]$. Thus $D$ has a finite covering
  \[
  K=\cup_{i=1}^N B_i \times [t_1,t_2], 
  \]
  where $B_i$ are balls, such that $\overline{B_i} \subset U$. Since
  $D\subset K$, we have $\bd_p K \subset U_{t_1,t_2} \setminus D$ and
  therefore $h < u+\eps$ on $\bd_p K$.  Let $u_\eps$ be the continuous
  weak solution with boundary values $u+\eps$ on $\bd_p K$. Then by
  the comparison principle $h\le u_\eps $in $K$. By Lemma \ref{super
    comparison} and Lemma \ref{comparison extends}, $u$ satisfies the
  comparison principle in $K$ and thus $u \le u_\eps$ in $K$. Now
  \[
  0 \le (h-u)_+ (h^m-u^m)_+ \le (u_\eps - u)(u_\eps^m - u^m) 
  \]
  and so by Lemma \ref{boundary value lemma}
\begin{align*}
  0\le \int_K (h-u)_+ (h^m-u^m)_+ \, dx \, dt &\le \int_K(u_\eps - u)(u_\eps^m - u^m) \, dx \, dt \\
&\le \eps |K| (\sup u + 1 ) + \eps |K| (\sup u + 1)^m.
\end{align*}
By construction of the set $D$, we have $h\le u+\eps$ in $U_{t_1,t_2} \setminus D$. Thus letting $\eps\rightarrow 0$ shows that $h\le u$ in $U_{t_1,t_2}$. We conclude that $u$ is $m$-superporous in $\Omega_T$.  
\end{proof}

The other nontrivial fact needed for Theorem \ref{thm: super} is that
locally bounded $m$-superporous functions are weak supersolutions.
For this purpose, we next present a Caccioppoli type estimate for the
weak supersolutions. 

\begin{lemma}\label{Caccioppoli}
  Let $u^m \in L^2(0,T;H^1(\Omega))$ be a weak supersolution, such that $u\le M$ in $\Omega_T$ for some $M>0$. Then
  \[
    \int_{\Omega_T} \zeta^2 | \nabla u^m|^2 \,dx \, dt \le 16M^{2m}T \int_\Omega |\nabla \zeta|^2 \,dx + 4M^{m+1} \int_\Omega \zeta^2 \, dx,
  \]
for every non-negative $\zeta \in C_0^\infty (\Omega)$. Note that $\zeta$ depends only on $x$. 
\end{lemma}

\begin{proof}
Formally, we use the test function $\ph=(M^m-u^m)\zeta^2$ in the definition of weak supersolutions. However, since no regularity for $u$ is assumed in the time variable, we need to use a time-regularized inequality to avoid the appearance of the possibly nonexistent quantity $u_t$. The proof is then just a straightforward computation. For the details, we refer to \cite[Lemma 2.15]{supersol}.
\end{proof}

The next step is to show that locally bounded $m$-superporous
functions are weak supersolutions. The idea of the proof is from
\cite{LC}: one approximates a given $m$-superporous function pointwise
by solutions to the obstacle problem. The approximants are weak
supersolutions, so the claim then follows from the Caccioppoli
estimate. For the PME, this argument has been carried out in
\cite{supersol}.

\begin{lemma}\label{$m$-superporous to weak super}
  Let $u$ be a locally bounded $m$-superporous function in $\Omega_T$. Then $u$ is a weak supersolution in $\Omega_T$.  
\end{lemma}
\begin{proof}
  We give the main points of the argument, referring to the proof of
  Theorem 3.2 in \cite{supersol} for the full details.  Since $u$ is
  lower semicontinuous, there exists a sequence of functions
  $\psi_k\in C^\infty(\Omega_T)$, such that
  \[
  \psi_i < \psi_{i+1} \quad \text{for every $i$,}
  \]
  and $\lim_{k\rightarrow \infty} \psi_k(x,t) = u(x,t)$ for every
  $(x,t)\in \Omega_T$. Without loss of generality, we may consider a
  set $Q_{t_1,t_2}\Subset \Omega_T$. For each $k$, let $u_k$ be the
  solution to the obstacle problem with obstacle function $\psi_k$. By
  Theorem 2.6 in \cite{obstacle}, a solution $u_k$ exist, such that
  \begin{align*}
    &u_k=\psi_k \text{ on } \bd_p Q_{t_1,t_2},\\
    &u_k\ge \psi_k \text{ in } Q_{t_1,t_2} \text{ and}\\
    &u_k^m \in L^2(t_1,t_2; H^1(Q)).
  \end{align*}
  Moreover, $u_k$ is a continuous weak supersolution in $Q_{t_1,t_2}$,
  and a weak solution in the open set $\{u_k>\psi_k\}$.
  The latter fact and the comparison
  principle of Remark \ref{finite union of cylinders} imply that 
    \begin{align*}
    u_1\le u_2\le \ldots \quad \text{and}\\
    u_k\le u \quad \text{for every $k$}.\\
  \end{align*}
  Now we have $u_k\rightarrow u$ due to the 
  inequalities
  \begin{equation*}
    \psi_k\leq u_k\leq u;
  \end{equation*}
  recall that $\lim_{k\rightarrow \infty} \psi_k(x,t) = u(x,t)$.



  Finally, the fact that $u$ is indeed a weak supersolution to the
  porous medium equation follows from the Caccioppoli estimate (Lemma
  \ref{Caccioppoli}) and weak compactness. 
\end{proof}

We now have everything we need to prove the second main result.

\begin{proof}[Proof of Theorem \ref{thm: super}]
  Let $u$ be a continuous weak supersolution in $\Omega_T$. By the
  definition of weak derivatives, it is clear that $u$ is also a very
  weak supersolution. 
  Let then $u$ be a continuous very weak supersolution. The comparison
  property with respect to continuous weak solutions for any
  space-time cylinder $U_{t_1,t_2} \Subset \Omega_T$ is the content of
  Lemma \ref{very weak super to $m$-superporous}. Thus continuous very
  weak supersolutions are $m$-superporous.
  Finally, a continuous $m$-superporous function $u$ is locally
  bounded, and hence a weak supersolution by Lemma
  \ref{$m$-superporous to weak super}.
\end{proof}

\vspace{0.5cm}
\noindent
\small{\textsc{Pekka Lehtel\"a},}
\small{\textsc{Department of Mathematics},}
\small{\textsc{P.O. Box 11100},}
\small{\textsc{FI-00076 Aalto University},}
\small{\textsc{Finland}}\\
\footnotesize{\texttt{pekka.lehtela@aalto.fi}}

\vspace{0.3cm}
\noindent
\small{\textsc{Teemu Lukkari},}
\small{\textsc{Department of Mathematics},}
\small{\textsc{P.O. Box 11100},}
\small{\textsc{FI-00076 Aalto University},}
\small{\textsc{Finland}}\\
\footnotesize{\texttt{teemu.lukkari@aalto.fi}}

\end{document}